\documentclass[journal,twoside,web]{ieeecolor}

\usepackage{mathdots}
\let\labelindent\relax
\usepackage{enumitem}
\usepackage{generic}
\usepackage{comment}
\usepackage{mathtools,amsmath}
\mathtoolsset{showonlyrefs}
\usepackage{tabularx} 
\usepackage{diagbox}
\usepackage{graphicx}
\usepackage{mathrsfs}
\usepackage{cite}
\usepackage{amssymb,amsfonts}
\usepackage{algorithmic}
\usepackage{tikz}
\usetikzlibrary{shapes,arrows,patterns,positioning}
\tikzstyle{block} = [draw, fill=gray!20, rectangle, 
    minimum height=2em, minimum width=4em]
\tikzstyle{sum} = [draw, fill=gray!20, circle, node distance=1.5cm]
\tikzstyle{input} = [coordinate]
\tikzstyle{output} = [coordinate]
\tikzstyle{pinstyle} = [pin edge={to-,thin,black}]

\usepackage{float}
\usepackage{accents}
\usepackage{multicol}
\usepackage{multirow}
\usepackage[ruled,vlined]{algorithm2e}

\usepackage{array,multirow}
\usepackage{hhline}
\usepackage{arydshln}
\usepackage[hidelinks=true]{hyperref}

\newtheorem{theorem}{Theorem}
\newtheorem{corollary}[theorem]{Corollary}
\newtheorem{lemma}[theorem]{Lemma}
\newtheorem{example}[theorem]{Example}

\newtheorem{proposition}[theorem]{Proposition}

\newtheorem{definition}[theorem]{Definition}

\DeclareMathOperator{\rank}{rank}

\DeclareMathOperator{\im}{im}
\DeclareMathOperator{\spec}{spec}

\newcommand{\set}[2]{\left\{#1 \mid #2\right\}}

\usepackage{textcomp}
\def\BibTeX{{\rm B\kern-.05em{\sc i\kern-.025em b}\kern-.08em
    T\kern-.1667em\lower.7ex\hbox{E}\kern-.125emX}}
\begin{document}
\title{A New Perspective on Willems' Fundamental Lemma: Universality of Persistently Exciting Inputs
}
\author{Amir Shakouri, Henk J. van Waarde, M. Kanat Camlibel
\thanks{The authors are with the Bernoulli Institute for Mathematics, Computer Science and Artificial Intelligence, University of Groningen, The Netherlands. Henk van Waarde acknowledges financial support by the Dutch Research Council under the NWO Talent Programme Veni Agreement (VI.Veni.22.335). Email:
        {\tt\small a.shakouri@rug.nl}, {\tt\small h.j.van.waarde@rug.nl}, {\tt\small m.k.camlibel@rug.nl}.}%
}

\maketitle
\thispagestyle{empty}

\begin{abstract}
In this letter, we provide new insight into Willems et al.'s fundamental lemma by studying the concept of universal inputs. An input is called \emph{universal} if, when applied to \emph{any} controllable system, it leads to input-output data that parametrizes all finite trajectories of the system. By the fundamental lemma, inputs that are persistently exciting of sufficiently high order are universal. The main contribution of this work is to prove the converse. Therefore, universality and persistency of excitation are equivalent.
\end{abstract}


\section{Introduction}

Willems et al.'s fundamental lemma \cite{willems2005note} states that all finite-length input-output trajectories of a controllable linear system can be parametrized using a single input-output trajectory generated by a persistently exciting sequence of inputs. This result plays an important role in system identification using subspace methods \cite{VERHAEGEN01111992}, where sufficiently rich input-output data can be used to identify the state-space system's parameters up to a similarity transformation. In addition, one of the most popular approaches to modern data-driven control is to employ the trajectory parametrization implied by the fundamental lemma in order to bypass explicit system identification. For instance, trajectory simulation \cite{markovsky2008data}, design of stabilizing feedback gains \cite{de2019formulas}, and construction of predictive controllers \cite{coulson2019data,berberich2020data} can be performed directly on the basis of input-output data. We refer the reader to \cite{van2020data,van2023informativity} for an alternative approach that avoids the use of persistently exciting input data. 

The fundamental lemma has been initially proven in \cite{willems2005note} within the context of behavioral theory. Other proofs of this lemma using state-space formulation have been presented in \cite{van2020willems,coulson2022quantitative,berberich2023quantitative}. Various extensions of this lemma have been studied in recent years. The extension of the fundamental lemma to multiple datasets has been studied in \cite{van2020willems} and to uncontrollable systems in \cite{yu2021controllability}. A robust version of this lemma for the case where the data are subject to noise is studied in \cite{coulson2022quantitative}. A counterpart of the fundamental lemma for linear parameter varying systems has been provided in \cite{verhoek2021fundamental}, and a version of this lemma in the frequency domain has been presented in \cite{meijer2023frequency}. Extensions of this lemma to stochastic systems \cite{pan2022stochastic}, 2D systems \cite{rapisarda2024input}, and several classes of nonlinear systems \cite{alsalti2023data,markovsky2022data,shang2024willems} are also among the recent works. In addition, counterparts of the fundamental lemma for continuous-time systems have been studied in \cite{rapisarda2023fundamental,lopez2024input}. 

The persistency of excitation condition imposes a lower bound on the required number of data samples. As shown for the first time in \cite{van2021beyond}, for a \emph{single} controllable system, one can improve the sample efficiency by generating an online experiment of shorter length that still enables the parametrization of all system trajectories. Moreover, \cite{camlibel2025shortest} constructed the shortest possible experiment for linear system identification, based on the necessary and sufficient conditions in \cite{camlibel2024beyond}. Nonetheless, the persistency of excitation condition guarantees a stronger property than the online input design. In fact, a persistently exciting input guarantees sufficiently rich data when applied to \emph{any} controllable system. In this sense, persistently exciting inputs are \emph{universal} as they work for the whole class of controllable systems. This is in contrast to online experiment design methods in \cite{van2021beyond,camlibel2025shortest} that are tailored to a \emph{specific} controllable system, because the online input design is guided by the outputs of the data-generating system. The existing literature lacks a full characterization of universal inputs. 

In this work, we investigate whether there exist other universal inputs than persistently exciting ones. In particular, we answer the following question: 
\begin{center}
    \emph{What are necessary and sufficient conditions \\ for an input to be universal?}
\end{center}
It turns out that persistency of excitation of a certain order is necessary and sufficient for the universality of the input. The sufficiency obviously follows from the fundamental lemma. The main contribution of this letter is to prove the necessity. To this end, given an input that is \emph{not} persistently exciting, we show how one can construct a controllable system and an initial condition such that the resulting input-output data \emph{do not} parametrize all finite-length trajectories of the system. The presented results provide insight into Willems et al.'s fundamental lemma by showing that persistency of excitation is not a conservative condition to guarantee universality. In the particular case of single-input systems, we further prove that if the input is not persistently exciting of a sufficiently high order, then \emph{almost any} controllable system can generate data that are not suitable for parametrization of the system trajectories. This also makes a connection to the recent work \cite{markovsky2023persistency}, where a similar observation was made.

The remainder of this letter is organized as follows. Section~\ref{sec:II} includes the problem statement. Section~\ref{sec:III} presents the main result of the letter. Section~\ref{sec:IV} is devoted to discussions on the main results. In Section~\ref{sec:V} we study numerical examples. Finally, Section~\ref{sec:VI} concludes the paper. 

\subsubsection*{Notation}

Let $\mathbb{Z}$, $\mathbb{Z}_+$ and $\mathbb{N}$ denote the sets of integers, non-negative integers, and positive integers, respectively. For $a,b \in \mathbb{Z}$ we define the integer interval \linebreak $[a,b] \coloneqq \set{c \in \mathbb{Z}}{a \leq c \leq b}$. Note that $[a,b]=\varnothing$ if $a>b$. Given $v:\mathbb{Z}_+\rightarrow\mathbb{R}^n$ and $T\in\mathbb{N}$, we define
\begin{equation}
v_{[0,T-1]}\coloneqq\begin{bmatrix}
v(0)^\top & v(1)^\top & \cdots & v(T-1)^\top
\end{bmatrix}^\top.
\end{equation}
The \emph{Hankel matrix} of $v_{[0,T-1]}$ of depth $k\in[1,T]$ is denoted by
\begin{equation}
\mathcal{H}_k(v_{[0,T-1]})\coloneqq \begin{bmatrix}
v(0) & v(1) & \cdots & v(T-k) \\
v(1) & v(2) & \cdots & v(T-k+1) \\
\vdots & \vdots & \ddots & \vdots \\
v(k-1) & v(k) & \cdots & v(T-1)
\end{bmatrix}.
\end{equation}

We say that $v_{[0,T-1]}$ is \emph{persistently exciting of order $k$} if $\mathcal{H}_k(v_{[0,T-1]})$ has full row rank\footnote{This definition of persistency of excitation is different from the one that typically appears in the adaptive control literature, defined for signals on an infinite time horizon (see, e.g., \cite[Def. 3.6.1]{mareels1996adaptive}).}.

We denote the kernel of matrix $M\in\mathbb{R}^{n\times m}$ by $\ker M\coloneqq \set{x\in\mathbb{R}^m}{Mx=0}$, and its image by \mbox{$\im M\coloneqq \set{Mx}{x\in\mathbb{R}^m}$}. The set of eigenvalues of a matrix $M\in\mathbb{R}^{n\times n}$ is denoted by $\spec M$. 

We say a matrix $M\in\mathbb{R}^{n\times n}$ is \emph{cyclic} if there exists $\zeta\in\mathbb{R}^n$ such that the pair $(M,\zeta)$ is controllable (see \cite[p. 67]{trentelmancontrol} and \cite[Def. 7.7.4]{bernstein2018scalar}).   

\section{Problem Statement}
\label{sec:II}

Let $n,m,p\in\mathbb{N}$. Consider the input-state-output system
\begin{subequations}
\label{eq:1}
\begin{align}
\label{eq:1(a)}
x(t+1)&=Ax(t)+Bu(t), \\
\label{eq:1(b)}
y(t)&=Cx(t)+Du(t),
\end{align}
\end{subequations}
where $u(t)\in\mathbb{R}^m$ is the input, $x(t)\in\mathbb{R}^n$ is the state, and $y(t)\in\mathbb{R}^p$ is the output. We identify system \eqref{eq:1} with the quadruple of matrices $(A,B,C,D)\in\mathcal{M}$, where
\begin{equation*}
\mathcal{M}\coloneqq \mathbb{R}^{n\times n}\times\mathbb{R}^{n\times m}\times\mathbb{R}^{p\times n}\times\mathbb{R}^{p\times m}.
\end{equation*}
We also define
\begin{equation}
\mathcal{M}_\text{cont}\coloneqq\set{(A,B,C,D)\in\mathcal{M}}{(A,B)\text{ is controllable}}.
\end{equation}
For an $(A,B,C,D)\in\mathcal{M}$, we denote its \emph{input-output behavior} by
\begin{equation}
\begin{split}
\mathfrak{B}(A,B,C,D)
\!\coloneqq\!\{(u,y)\!:\!\mathbb{Z}_+\!\rightarrow\! \mathbb{R}^m\!\times\!\mathbb{R}^p\mid\exists x:\mathbb{Z}_+\!\rightarrow\!\mathbb{R}^n\\
\text{ such that }\eqref{eq:1(a)}\text{ and }\eqref{eq:1(b)}\text{ hold}\},
\end{split}
\end{equation}
and its \emph{$k$--restricted input-output behavior} by
\begin{equation}
\mathfrak{B}_k(A,B,C,D)\!\coloneqq\!\set{\!\begin{bmatrix}
u_{[0,k-1]} \\ y_{[0,k-1]}
\end{bmatrix}}{(u,y)\in\mathfrak{B}(A,B,C,D)\!}\!.
\end{equation}
In addition, we define the \emph{$k$--restricted input-state behavior} of \eqref{eq:1(a)} as
\begin{equation}
\mathfrak{B}_k(A,B)\coloneqq\mathfrak{B}_k(A,B,I_n,0).
\end{equation}

For any trajectory $\begin{bmatrix}
u_{[0,T-1]} \\ y_{[0,T-1]}
\end{bmatrix}\in\mathfrak{B}_T(A,B,C,D)$, it is evident that
\begin{equation}
\label{eq:th:1-b1_before}
\im \begin{bmatrix}
\mathcal{H}_k(u_{[0,T-1]}) \\
\mathcal{H}_k(y_{[0,T-1]})\end{bmatrix}\subseteq\mathfrak{B}_k(A,B,C,D)
\end{equation}
for all $k\in[1,T]$. The celebrated fundamental lemma states that, under suitable conditions, \eqref{eq:th:1-b1_before} holds with equality (see \mbox{\cite[Thm. 1]{willems2005note}} and \cite[Thm. 1]{van2020willems}).

\begin{proposition}[Willems et al.'s fundamental lemma]
\label{prop:1}
Let $T\in\mathbb{N}$, $L\in[1,T]$, and $(A,B,C,D)\in\mathcal{M}_\text{cont}$. If $u_{[0,T-1]}$ is persistently exciting of order $n+L$, then the following statements hold:
\begin{enumerate}[label=(\alph*),ref=\ref{prop:1}(\alph*)]
    \item\label{prop:1(a)} We have
    \begin{equation}
    \label{eq:th:1-a1}
    \rank \begin{bmatrix}
\mathcal{H}_L(u_{[0,T-1]}) \\
\mathcal{H}_1(x_{[0,T-L]}) 
\end{bmatrix}=n+Lm
    \end{equation}
    for all $x_{[0,T-L]}$ satisfying
    \begin{equation}
    \label{eq:th:1-a2}
    \begin{bmatrix}
u_{[0,T-L]} \\ x_{[0,T-L]}
\end{bmatrix}\in\mathfrak{B}_{T-L+1}(A,B).
    \end{equation}
\item\label{prop:1(b)} We have
\begin{equation}
\label{eq:th:1-b1}
\mathfrak{B}_L(A,B,C,D)=\im \begin{bmatrix}
\mathcal{H}_L(u_{[0,T-1]}) \\
\mathcal{H}_L(y_{[0,T-1]}) 
\end{bmatrix}
\end{equation}
for all $y_{[0,T-1]}$ satisfying
\begin{equation}
\label{eq:th:1-b2}
\begin{bmatrix}
u_{[0,T-1]} \\ y_{[0,T-1]}
\end{bmatrix}\in\mathfrak{B}_T(A,B,C,D).
\end{equation}
\end{enumerate}
\end{proposition}

Persistency of excitation of order $n+L$ requires a sufficiently long trajectory, namely
\begin{equation}
T\geq (n+L)(m+1)-1.
\end{equation}
Based on Proposition~\ref{prop:1}, for a \emph{single} controllable system, persistency of excitation of order $n+L$ is a sufficient condition to determine the $L$--restricted behavior of the system from data, via \eqref{eq:th:1-b1}. However, the following example shows that this condition is not necessary. 
\begin{example}
\label{ex:1}
Consider the system given by
\begin{equation}
A=\begin{bmatrix}
0 & 1 \\
0 & 0
\end{bmatrix},\ B=\begin{bmatrix}
0 \\ 1
\end{bmatrix},\ C=\begin{bmatrix}
1 & 0
\end{bmatrix},\ \text{and}\ D=0.
\end{equation}
Let $x(0)=\begin{bmatrix}
x_1(0) & x_2(0)
\end{bmatrix}^\top$, $T=3$, $L=1$, $u(0)= 1$, and $u(1)=u(2)=0$. We have
\begin{equation}
\label{eq:ex1-1}
\begin{bmatrix}
\mathcal{H}_1(u_{[0,2]}) \\
\mathcal{H}_1(y_{[0,2]}) 
\end{bmatrix}
=\begin{bmatrix}
1   & 0   & 0 \\
x_1(0) & x_2(0) & 1 
\end{bmatrix},
\end{equation}
which has full row rank for all $x(0)\in\mathbb{R}^2$. Therefore, 
\begin{equation}
\im \begin{bmatrix}
\mathcal{H}_1(u_{[0,2]}) \\
\mathcal{H}_1(y_{[0,2]}) 
\end{bmatrix}=\mathbb{R}^2.
\end{equation}
By \eqref{eq:th:1-b1_before}, this implies \eqref{eq:th:1-b1}. However, the given input sequence is not persistently exciting of order $3$. 
\end{example}

Although persistency of excitation is not necessary to guarantee \eqref{eq:th:1-b1} for a \emph{single} system, in disguise, Proposition~\ref{prop:1} makes a statement about \emph{all} controllable systems. Indeed, it implies that persistently exciting inputs of order $n+L$ are \emph{universal} in the sense that they guarantee statement \ref{prop:1(b)} \emph{for all} $(A,B,C,D)\in\mathcal{M}_\text{cont}$. A natural question to ask is whether there are other universal inputs than persistently exciting ones of order $n+L$. To formalize this question, we introduce the notion of a \emph{universal input}. 

\begin{definition}
\label{def:1}
Let $T\in\mathbb{N}$ and $L\in[1,T]$. An input  $u_{[0,T-1]}$ is called \emph{universal for determining the $L$--restricted behavior} if for every $(A,B,C,D)\in\mathcal{M}_\text{cont}$ and every $y_{[0,T-1]}$ satisfying $\begin{bmatrix}
u_{[0,T-1]} \\ y_{[0,T-1]}
\end{bmatrix}\in\mathfrak{B}_T(A,B,C,D)$ we have 
\begin{equation}
\mathfrak{B}_L(A,B,C,D)=\im \begin{bmatrix}
\mathcal{H}_L(u_{[0,T-1]}) \\
\mathcal{H}_L(y_{[0,T-1]}) 
\end{bmatrix}.
\end{equation}
\end{definition}\vspace{0.25 cm}

Proposition~\ref{prop:1(b)} implies that an input signal is universal for determining the $L$--restricted behavior if it is persistently exciting of order $n+L$. In this letter, we investigate necessary and sufficient conditions under which a sequence of inputs is universal in the sense of Definition~\ref{def:1}.

\section{Main Results}
\label{sec:III}

The following theorem presents the main result of this letter, showing that universality and persistency of excitation are equivalent.

\begin{theorem}
\label{th:2}
Let $T\in\mathbb{N}$ and $L\in[1,T]$. An input $u_{[0,T-1]}$ is universal for determining the $L$--restricted behavior \emph{if and only if} it is persistently exciting of order $n+L$.  
\end{theorem}

To prove Theorem~\ref{th:2}, we need an auxiliary result presented in the following lemma, which is interesting on its own. This lemma shows that for any input that is not persistently exciting of order $n+L$, there exist a controllable system and an initial state such that the Hankel matrix of input-state data is rank-deficient, i.e., \eqref{eq:th:1-a1} does not hold\footnote{The proof of this lemma is inspired by the idea of \cite[Lem. 24]{camlibel2024beyond}.}. 

\begin{lemma}
\label{lem:henk}
Let $T\in\mathbb{N}$, $L\in[1,T]$, and $u_{[0,T-1]}\in\mathbb{R}^{mT}$. If $u_{[0,T-1]}$ is not persistently exciting of order $n+L$, then there exists a controllable pair $(A,B)$ and a state $x_{[0,T-L]}$ satisfying $\begin{bmatrix}
u_{[0,T-L]} \\ x_{[0,T-L]}
\end{bmatrix}\in\mathfrak{B}_{T-L+1}(A,B)$ such that
\begin{equation}
\label{eq:lem:henk-1}
\begin{bmatrix}
v^\top & w^\top
\end{bmatrix}\begin{bmatrix}
\mathcal{H}_L(u_{[0,T-1]}) \\
\mathcal{H}_1(x_{[0,T-L]}) 
\end{bmatrix}=0
\end{equation}
for some $v\in\mathbb{R}^{mL}$ and $w\in\mathbb{R}^n\backslash\{0\}$.
\end{lemma}

\begin{proof}
We distinguish two cases: $T<n+L-1$ and \mbox{$T\geq n+L-1$}. In case $T<n+L-1$, we have
\begin{equation}
\rank\mathcal{H}_1(x_{[0,T-L]})\leq \min\{n,T-L+1\}<n.
\end{equation}
This implies that \eqref{eq:lem:henk-1} holds with $v=0$ and a nonzero \mbox{$w\in\ker\mathcal{H}_1(x_{[0,T-L]})^\top$}. Next, we consider the case \mbox{$T\geq n+L-1$}. Let $\eta\in\ker \mathcal{H}_{n+L}(u_{[0,T-1]})^\top\backslash\{0\}$.
Define $\eta_0,\ldots,\eta_{n+L-1}\in\mathbb{R}^m$ by $\begin{bmatrix}
\eta_0^\top & \cdots & \eta_{n+L-1}^\top
\end{bmatrix}^\top=\eta$. Observe that for every $t\in[0,T-n-L]$ we have
\begin{equation}
\label{eq:ker2}
\sum_{i=0}^{n+L-1}\eta_i^\top u(t+i)=0.
\end{equation}
Now, we construct a controllable pair $(A,B)$ and an initial condition $x(0)$ such that the generated input-state trajectory satisfies \eqref{eq:lem:henk-1} for some $v\in\mathbb{R}^{mL}$ and $w\in\mathbb{R}^n\backslash\{0\}$. To this end, define
\begin{equation}
\label{eq:Lambda}
\Lambda(\eta)\coloneqq\set{\lambda\in\mathbb{C}}{\textstyle\sum_{i=0}^{n+L-1}\lambda^i \eta_i=0}.
\end{equation}
We note that the set $\Lambda(\eta)$ is finite. Let $A\in\mathbb{R}^{n\times n}$ be a cyclic matrix satisfying\footnote{For example, a Jordan block with its eigenvalue not in $\Lambda(\eta)$.} $\spec A\cap\Lambda(\eta)=\varnothing$. Let $\zeta\in\mathbb{R}^n$ be such that the pair $(A,\zeta)$ is controllable. Define matrices $E_i$, \mbox{$i\in[-1,n+L-1]$}, by the recursion
\begin{equation}
\label{eq:Ei_recursive}
E_{i-1}=AE_{i}+\zeta\eta_i^\top,
\end{equation}
where $E_{n+L-1}=0$. Let $B=E_{-1}=\sum_{i=0}^{n+L-1}A^i\zeta\eta_i^\top$. We claim that $(A,B)$ is controllable. To see this, let\linebreak $z\in\mathbb{C}^n\backslash \{0\}$ be a left eigenvector of $A$, i.e., $z^*A=\lambda z^*$ for some $\lambda\in\spec A$. By the Hautus test and controllability of the pair $(A,\zeta)$, this implies that $z^*\zeta\neq 0$. This in turn implies that $z^* B=z^* \zeta\sum_{i=0}^{n+L-1} \lambda^i \eta_i^\top\neq 0$, thus the pair $(A,B)$ is controllable. Take
\begin{equation}
\label{eq:x(0)}
x(0)=-\sum_{i=0}^{n+L-2} E_i u(i),
\end{equation}
and consider the state trajectory of $(A,B)$, starting from $x(0)$, given by
\begin{equation}
\label{eq:state_construction}
x(t+1)=Ax(t)+Bu(t),\ \hspace{0.25 cm} t\in[0,T-L].
\end{equation}
We claim that this trajectory obeys the following formula with $s\coloneqq T-L-n+1$:
\begin{subequations}
\label{eq:statetraj}
\begin{align}
\label{eq:statetraj-1}
\hspace{-0.75 cm} x(t)&=-\sum_{i=0}^{n+L-2} E_i u(t+i)\hspace{0.25 cm} \text{if}\hspace{0.25 cm} t\in[0,s], \\
\hspace{-0.75 cm} x(t+s)&=\sum_{j=0}^{t-1}\sum_{i=j}^{n+L-2} A^{t-j-1}\zeta\eta_{i-j}^\top u(i+s)\\
\label{eq:statetraj-2}
&-\sum_{i=t}^{n+L-2}E_{i-t} u(i+s)  \hspace{0.25 cm} \text{if}\hspace{0.25 cm} t\in[1,n-1].
\end{align}
\end{subequations}
We use induction to prove \eqref{eq:statetraj-1} and \eqref{eq:statetraj-2}. Note that \eqref{eq:statetraj-1} holds for $t=0$ in view of \eqref{eq:x(0)}. Assume that \eqref{eq:statetraj-1} holds for some $t\in[0,s-1]$. Then, using \eqref{eq:state_construction} it follows that
\begin{subequations}
\mathtoolsset{showonlyrefs=false}
\label{eq:moved}
\begin{align}
&x(t+1)=-A\sum_{i=0}^{n+L-2} E_i u(t+i)+E_{-1}u(t) \\
& \stackrel{\eqref{eq:Ei_recursive}}{=} -\sum_{i=1}^{n+L-2} E_{i-1} u(t+i)+\sum_{i=0}^{n+L-2} \zeta\eta_{i}^\top u(t+i).
\end{align}
\mathtoolsset{showonlyrefs=true}
\end{subequations}
From \eqref{eq:moved} we have
\begin{equation}
\begin{split}
x(t+1)&\stackrel{\eqref{eq:ker2}}{=} -\!\!\!\sum_{i=0}^{n+L-3} \!\! E_{i} u(t\!+\!1\!+\!i)-\zeta\eta_{n+L-1}^\top u(t\!+\!n\!+\!L\!-\!1) \\
&\stackrel{\eqref{eq:Ei_recursive}}{=} -\sum_{i=0}^{n+L-2} E_{i} u(t+1+i),
\end{split}
\end{equation}
which implies that \eqref{eq:statetraj-1} is satisfied for $t+1$. This proves \eqref{eq:statetraj-1}. 

Next, we prove \eqref{eq:statetraj-2}. First, note that \eqref{eq:moved} holds for $t=s$ due to \eqref{eq:state_construction} and \eqref{eq:statetraj-1}. Thus, \eqref{eq:statetraj-2} holds for $t=1$. Assume that \eqref{eq:statetraj-2} holds for some $t\in[1,n-2]$. We observe that using \eqref{eq:state_construction}, we have
\begin{equation}
\label{eq:sth}
\begin{split}
&x(t+s+1)=-\sum_{i=t}^{n+L-2}AE_{i-t} u(s+i) \\ &+\sum_{j=0}^{t-1}\sum_{i=j}^{n+L-2} A^{t-j}\zeta\eta_{i-j}^\top u(s+i)+E_{-1}u(t+s).
\end{split}
\end{equation}
Using \eqref{eq:Ei_recursive}, \eqref{eq:sth} can be reduced to
\begin{equation}
\begin{split}
&x(t+s+1)=-\sum_{i=t+1}^{n+L-2}E_{i-t-1} u(s+i) \\ &+\sum_{i=t}^{n+L-2}\zeta\eta_{i-t} u(s+i)+\sum_{j=0}^{t-1}\sum_{i=j}^{n+L-2} A^{t-j}\zeta\eta_{i-j}^\top u(s+i) \\
&=\!-\!\!\sum_{i=t+1}^{n+L-2}\!E_{i-t-1} u(s\!+\!i)+\sum_{j=0}^{t}\sum_{i=j}^{n+L-2}\! A^{t-j}\zeta\eta_{i-j}^\top u(s\!+\!i),
\end{split}
\end{equation}
which implies that \eqref{eq:statetraj-2} is satisfied for $t+1$. This proves that \eqref{eq:statetraj-2} holds. Now, let $\xi\neq 0$ be such that $\xi^\top A^i \zeta=0$ for all $i\in[0,n-2]$. Observe from \eqref{eq:Ei_recursive} that for every $i\in[L,n+L-1]$ we have $\xi^\top E_i=0$. Thus, we have
\begin{equation}
\label{eq:final_arg1}
\xi^\top\sum_{i=0}^{n+L-2} E_i u(t+i)=\xi^\top\sum_{i=0}^{L-1} E_i u(t+i),
\end{equation}
and for every $t\in[1,n-1]$ we have
\begin{equation}
\label{eq:final_arg2}
\xi^\top\sum_{i=t}^{n+L-2}E_{i-t} u(i+s)=\xi^\top\sum_{i=0}^{L-1} E_i u(i+t+s)
\end{equation}
\begin{equation}
\label{eq:final_arg3}
\text{and }\ \xi^\top\sum_{j=0}^{t-1}\sum_{i=j}^{n+L-2} \!\!A^{t-j-1}\zeta\eta_{i-j}^\top u(i+s)=0.
\end{equation}
Based on \eqref{eq:statetraj-1} and \eqref{eq:statetraj-2}, one can verify using \eqref{eq:final_arg1}, \eqref{eq:final_arg2}, and \eqref{eq:final_arg3} that for every $t\in[0,T-L]$ we have \mbox{$\xi^\top x(t)=-\xi^\top \sum_{i=0}^{L-1} E_i u(t+i)$}. Therefore, \eqref{eq:lem:henk-1} holds with $v=\begin{bmatrix}
E_0 & \cdots & E_{L-1}
\end{bmatrix}^\top \xi$ and $w=\xi$.
\end{proof}

Now, we use Lemma~\ref{lem:henk} to prove Theorem~\ref{th:2}. 

\textit{Proof of Theorem~\ref{th:2}:} 
The ``if'' part follows from Proposition~\ref{prop:1(b)}. For the ``only if'' part, assume that the input $u_{[0,T-1]}$ is not persistently exciting of order $n+L$. Based on Lemma~\ref{lem:henk}, this implies that there exists a controllable pair $(A,B)$ and $x_{[0,T-L]}$ satisfying \eqref{eq:th:1-a2} such that \eqref{eq:lem:henk-1} holds for some $v\in\mathbb{R}^{mL}$ and $w\in\mathbb{R}^n\backslash\{0\}$. Let $C\in\mathbb{R}^{p\times n}$ be such that its first row equals $w^\top$. Also, let $D=0$ and \mbox{$y_{[0,T-1]}=Cx_{[0,T-1]}$}. Observe that $y_{[0,T-1]}$ satisfies \eqref{eq:th:1-b2}. Define $e_1\coloneqq\begin{bmatrix}
1 & 0 & \cdots & 0
\end{bmatrix}^\top\in\mathbb{R}^p$ and verify that
\begin{equation}
\label{eq:th2_pf1}
\begin{bmatrix}
v^\top \!\!&\!\! e_1^\top \!\!&\!\! 0
\end{bmatrix}\!\!\begin{bmatrix}
\mathcal{H}_L(u_{[0,T-1]}) \\ \mathcal{H}_L(y_{[0,T-1]})
\end{bmatrix}\!=\!\begin{bmatrix}
v^\top \!\!&\!\! w^\top 
\end{bmatrix}\!\!\begin{bmatrix}
\mathcal{H}_L(u_{[0,T-1]}) \\ \mathcal{H}_1(x_{[0,T-L]})
\end{bmatrix}\!\!=\!0.
\end{equation}
To prove that \eqref{eq:th:1-b1} does not hold, it suffices to show that for some $\begin{bmatrix}
\bar{u}_{[0,L-1]} \\ \bar{y}_{[0,L-1]}
\end{bmatrix}\in\mathfrak{B}_L(A,B,C,D)$ we have
\begin{equation}
\label{eq:th2_pf1p}
\begin{bmatrix}
v^\top & e_1^\top & 0 
\end{bmatrix}\begin{bmatrix}
\bar{u}_{[0,L-1]} \\ \bar{y}_{[0,L-1]}
\end{bmatrix}\neq 0.
\end{equation}
To this end, take $\bar{u}(t)=0$ for all $t\in[0,L-1]$ and $\bar{x}(0)\in\mathbb{R}^n$ such that $w^\top \bar{x}(0)\neq 0$. Let $\bar{y}_{[0,L-1]}$ be the output trajectory generated by $(A,B,C,D)$ with initial state $\bar{x}(0)$ and input $\bar{u}_{[0,L-1]}$. Now, one can observe that \eqref{eq:th2_pf1p} holds since
\begin{equation}
\begin{bmatrix}
v^\top & e_1^\top & 0 
\end{bmatrix}\begin{bmatrix}
\bar{u}_{[0,L-1]} \\ \bar{y}_{[0,L-1]}
\end{bmatrix}=e_1^\top \bar{y}(0)=w^\top \bar{x}(0)\neq 0.
\end{equation}
Therefore, \eqref{eq:th:1-b1} does not hold. \hfill \QED

\section{Discussion and Further Results}
\label{sec:IV}

In this section, we discuss connections between our results and those in the literature. 

\subsection{Persistency of excitation of order $n$}

A corollary of the fundamental lemma \cite[Cor. 2(i)]{willems2005note} (also see \cite[Cor. 3(1)]{markovsky2021behavioral}) asserts that if the input $u_{[0,T-1]}$ is persistently exciting of order $n$, then $T$--length state trajectories generated by any controllable system satisfy $\rank\mathcal{H}_1(x_{[0,T-1]})=n$. This result is not fully correct due to a small error in the time horizon. We state the correct version of \cite[Cor. 2(i)]{willems2005note} as the following proposition. The proof of this result is omitted here since it follows similar lines as the proof of \cite[Thm. 1]{van2020willems}. 

\begin{proposition}
\label{prop:cor}
Let $T\in\mathbb{N}$, $u_{[0,T]}\in\mathbb{R}^{m(T+1)}$, and $(A,B,C,D)\in\mathcal{M}_\text{cont}$. If $u_{[0,T-1]}$ is persistently exciting of order $n$, then $\mathcal{H}_1(x_{[0,T]})$ has full row rank for all $x_{[0,T]}$ satisfying
\begin{equation}
\begin{bmatrix}
\label{eq:prop:cor}
u_{[0,T]} \\ x_{[0,T]}
\end{bmatrix}\in\mathfrak{B}_{T+1}(A,B).
\end{equation}%

\end{proposition}

Now, using the results of this letter, we can show that persistency of excitation of order $n$ is not only sufficient but also \emph{necessary} to guarantee that the state trajectory $x_{[0,T-1]}$ generated by any controllable system satisfies \mbox{$\rank\mathcal{H}_1(x_{[0,T-1]})=n$}. 

\begin{corollary}
Let $T\in\mathbb{N}$ and $u_{[0,T]}\in\mathbb{R}^{m(T+1)}$. Then, $\mathcal{H}_1(x_{[0,T]})$ has full row rank for all $x_{[0,T]}$ satisfying \eqref{eq:prop:cor} and all controllable $(A,B)\in\mathbb{R}^{n\times n}\times \mathbb{R}^{n\times m}$ if and only if $u_{[0,T-1]}$ is persistently exciting of order $n$. 
\end{corollary}
\begin{proof}
The ``if'' part is a direct consequence of Proposition~\ref{prop:cor}. One can prove the ``only if'' part by taking $L=0$, and $(A,B)$, $x(0)$, and $\xi$ as in the proof of Lemma~\ref{lem:henk}. Then, one can observe that $\xi^\top \mathcal{H}_1(x_{[0,T]})=0$.
\end{proof}

\subsection{Single-input case}

For the single-input case, $m=1$, we can sharpen Lemma~\ref{lem:henk} by showing that for \emph{almost any} controllable pair $(A,B)$, there exists $x(0)\in\mathbb{R}^n$ for which the Hankel matrix of the input-state data is rank deficient. A similar observation has been made in \cite[Thm. 3]{markovsky2023persistency}. This result is presented in the following proposition. We recall that the set $\Lambda(\eta)$ has been defined in~\eqref{eq:Lambda}.

\begin{proposition}
\label{prop:m1}
Let $m=1$, $T\in\mathbb{N}$, $L\in[1,T]$, $u_{[0,T-1]}\in~\mathbb{R}^{T}$, and $\eta\in\mathbb{R}^{n+L}\backslash\{0\}$ be such that $\eta^\top \mathcal{H}_{n+L}(u_{[0,T-1]})=0$. Then, for any $(A,B)$ such that $\spec A\cap\Lambda(\eta)=\varnothing$ there exists a state $x_{[0,T-L]}$, satisfying~\eqref{eq:th:1-a2}, such that
\begin{equation}
\label{eq:rank_deficient}
\rank \begin{bmatrix}
\mathcal{H}_L(u_{[0,T-1]}) \\
\mathcal{H}_1(x_{[0,T-L]}) 
\end{bmatrix}<n+L.
\end{equation}
\end{proposition}\vspace{0.25 cm}
\begin{proof}
Let $A\in\mathbb{R}^{n\times n}$ be such that $\spec A\cap\Lambda(\eta)=\varnothing$ and $B\in\mathbb{R}^{n}$. Note that $\spec A\cap\Lambda(\eta)=\varnothing$ implies that the matrix $\sum_{i=0}^{n+L-1} A^i\eta_i$ is invertible. Let $\zeta=(\sum_{i=0}^{n+L-1} A^i\eta_i)^{-1}B$. Also, let $E_i$ for $i\in[0,n+L-1]$ and $x(0)$ be as in the proof of Lemma~\ref{lem:henk}. The state trajectory generated by $(A,B)$ starting from $x(0)$ satisfies \eqref{eq:statetraj-1} and \eqref{eq:statetraj-2}. Now, let $\xi\in\mathbb{R}^n\backslash\{0\}$ be such that $\xi^\top A^i \zeta=0$ for all $i\in[0,n-2]$. This implies that for every $i\in[L,n+L-1]$ we have $\xi^\top E_i=0$. Thus, based on \eqref{eq:statetraj-1} and \eqref{eq:statetraj-2}, one can verify that for every $t\in[0,T-L]$ we have $\xi^\top x(t)=-\xi^\top \sum_{i=0}^{L-1} E_i u(t+i)$. Therefore, \eqref{eq:lem:henk-1} holds with $v=\begin{bmatrix}
E_0 & \cdots & E_{L-1}
\end{bmatrix}^\top \xi$ and $w=\xi$.
\end{proof}

We note that $\Lambda(\eta)$ is a finite set since $\eta\neq 0$. This implies that the set of matrices $A\in\mathbb{R}^{n\times n}$ satisfying \linebreak $\spec A\cap\Lambda(\eta)=\varnothing$ is dense in $\mathbb{R}^{n\times n}$. Hence, Proposition~\ref{prop:m1} shows that, in the single-input case, if the input signal is not persistently exciting of order $n+L$, then \emph{almost any} controllable system can generate rank-deficient data, i.e., data satisfying \eqref{eq:rank_deficient}. As we will demonstrate in Example~\ref{ex:sec:V-1}, this is generally not true for multi-input systems. 

\subsection{Identification of the state-space system's parameters}

Based on Willems et al.'s fundamental lemma, the controllability of $(A,B)$, along with persistency of excitation of the input of order $n+L$, guarantees that the generated input-output data determines the $L$--restricted behavior of the system. For this, one does not need the pair $(C,A)$ to be observable. However, if $(C,A)$ is observable and $L$ is strictly larger than the so-called \emph{lag} of the system, then the state-space system's parameters $(A,B,C,D)$ can be recovered from the behavior $\mathfrak{B}_L(A,B,C,D)$ up to a similarity transformation (see \cite[Thm. 8.16]{markovsky2006exact}). Therefore, for controllable and observable systems with lag $\ell$, inputs that are universal for determining the $(\ell+1)$--restricted behavior enable system identification up to a similarity transformation. This can be accomplished, e.g., using the subspace identification methods in \cite[Ch. 9]{VERHAEGEN01111992}.

\section{Numerical Examples}
\label{sec:V}
 
In this section, we present two numerical examples. For both examples, given an input signal that is not persistently exciting of order $n+L$, we use a similar procedure as in the proof of Lemma~\ref{lem:henk} to construct controllable systems capable of generating rank-deficient data, i.e., data not satisfying the rank condition \eqref{eq:th:1-a1}\footnote{For these examples, we use the MATLAB code available at \href{https://github.com/a-shakouri/universal-inputs-and-persistency-of-excitation}{https://github.com/a-shakouri/universal-inputs-and-persistency-of-excitation}.}. 

\begin{example}
\label{ex:sec:V-2}
Let $n=3$, $m=2$, $L=1$, and $T=8$. Consider the input signal given in Table~\ref{table:2}. This input is not persistently exciting of order $4$. The corresponding values of $\eta_i$, $i\in[0,3]$, satisfying \eqref{eq:ker2}, are computed as follows:
\begin{equation}
\begin{bmatrix}
\eta_0 \!\!&\!\! \eta_1 \!\!&\!\! \eta_2 \!\!&\!\! \eta_3
\end{bmatrix}\!\!=\!\!\begin{bmatrix}
-0.2039 & -0.1162 & 0.0598 & -0.0782 \\
0.4924 & 0.4964 & 0.6281 & 0.2277
\end{bmatrix}\!.
\end{equation}
We choose a matrix $A$ and a vector $\zeta$, with the entries drawn uniformly at random from intervals $[-3,3]$ and $[-1,1]$, respectively, as follows:
\begin{equation}
A=\begin{bmatrix}
-0.2675 & -1.3834 & -1.4581 \\
-0.2088 & -2.2509 & -2.5228 \\
2.3349 &  -1.5816 & 1.0206
\end{bmatrix},\hspace{0.25 cm} \zeta=\begin{bmatrix}
-0.4351 \\ 0.0482 \\ 0.8496
\end{bmatrix}.
\end{equation}
The pair $(A,\zeta)$ is controllable. Next, we use the recursion \eqref{eq:Ei_recursive} to obtain $E_{i}$ for $i\in[-1,2]$ as follows:
\begin{equation}
\begin{split}
&E_2\!=\!\begin{bmatrix}
0.0340 & -0.0991 \\
-0.0038 & 0.0110 \\
-0.0664 & 0.1935
\end{bmatrix}\!\!,\ E_1=\begin{bmatrix}
0.0669 & -0.5441 \\
0.1718 & -0.4618 \\
0.0684 & 0.4823
\end{bmatrix}\!\!, \\
&E_0\!=\!\begin{bmatrix}
-0.3047 & -0.1349 \\
-0.5788 & -0.0398 \\
-0.1444 & 0.3741
\end{bmatrix}\!\!,\ E_{-1}\!=\!\begin{bmatrix}
1.1815 & -0.6687 \\
1.7209 & -0.8023 \\
-0.1166 &  0.5480
\end{bmatrix}\!\!.
\end{split}
\end{equation}
We let $B=E_{-1}$ and compute $x(0)$ from \eqref{eq:x(0)} to have
\begin{equation}
x(0)=\begin{bmatrix} -0.461 & -0.3879 & 0.0665 \end{bmatrix}^\top.
\end{equation}
We also compute $\xi^\top=\begin{bmatrix}
0 & 0 & 1
\end{bmatrix}\begin{bmatrix}
\zeta & A\zeta & A^{2}\zeta
\end{bmatrix}^{-1}$ as
\begin{equation}
\xi=\begin{bmatrix}
2.4303 & -1.4758 & 1.3285
\end{bmatrix}^\top.
\end{equation}
Now, the state trajectory generated by $(A,B)$ with initial condition $x(0)$ and input $u_{[0,7]}$ is shown in Table~\ref{table:3}. One can check that $\begin{bmatrix}
I \\ E_0^\top 
\end{bmatrix}\xi \in\ker \begin{bmatrix}
\mathcal{H}_1(u_{[0,7]}) \\
\mathcal{H}_1(x_{[0,7]}) 
\end{bmatrix}^\top$, which implies that $\rank \begin{bmatrix}
\mathcal{H}_1(u_{[0,7]}) \\
\mathcal{H}_1(x_{[0,7]}) 
\end{bmatrix}<5$. 
\end{example}
\begin{table}[h]
    \centering
    \caption{Input signal for Example~\ref{ex:sec:V-2}.}
    \begin{tabular}{c|cccccccc}
       $t$  & 0 & 1 & 2 & 3 & 4 & 5 & 6 & 7 \\ \hline
       \multirow{2}{*}{\!$u(t)$\!} & -0.46 & -1.09 & 0.9  & 1.38 & 1.24  & 1.54 & 0.22 & -1.68 \\
              & 1.86  &  -0.87  & -1.42 & 1.06 & 1.6 & -1.94 & 1.11 & -1.03
    \end{tabular}
    \label{table:2}
\end{table}

\begin{table*}[h]
    \centering
    \caption{State trajectory for Example~\ref{ex:sec:V-2}.}
    \begin{tabular}{c|cccccccc}
       $t$  & 0 & 1 & 2 & 3 & 4 & 5 & 6 & 7 \\ \hline
       \multirow{3}{*}{$x(t)$} & -0.4610 & -1.2243  &  0.6834  &  1.1065 & -0.4624 &   0.7518  &  0.4177  &  1.0417 \\
              & -0.3879 &  -1.4824  &  0.7041 &   1.3940 & -0.3905  &  1.2941 &  -0.0876  &  2.2183 \\
              & 0.0665  &  0.6780 & -0.1718 &  -0.5763 &   0.2108  &  0.4853 &  -1.0387  &  0.6362
    \end{tabular}
    \label{table:3}
\end{table*}

\begin{example}
\label{ex:sec:V-1}
For the sake of illustration, we consider a single-state multi-input system with $n=1$ and $m=2$, as
\begin{equation}
x(t+1)=ax(t)+\begin{bmatrix}
b_1 & b_2
\end{bmatrix}u(t).
\end{equation}
Let $T=7$, $L=2$, and the input signal be given in Table~\ref{table:1}. This input is not persistently exciting of order $3$. Fig.~\ref{fig:1} shows a set of $10000$ systems that are capable of generating rank-deficient data for such an input. For instance, the trajectory generated by system $a=-0.9262$, $b_1=-0.3273$, and \linebreak $b_2=-0.3356$, indicated by the red dot in Fig.~\ref{fig:1}, starting from $x(0)= 0.5561$, is such that $\rank\begin{bmatrix}
\mathcal{H}_2(u_{[0,6]}) \\
\mathcal{H}_1(x_{[0,5]}) 
\end{bmatrix}=4$. 
\end{example}

\begin{table}[h]
    \centering
    \caption{Input signal for Example~\ref{ex:sec:V-1}.}
    \begin{tabular}{c|ccccccc}
       $t$  & 0 & 1 & 2 & 3 & 4 & 5 & 6 \\ \hline
       \multirow{2}{*}{$u(t)$} & -1.24 & 0.35 & -0.56 & 1.24 & -1.66 & 0.61 & -0.39 \\
        & 0.67 & 0.7 & 0.48 & -1.92 & 1.9 & -1.08 &  -1.51
    \end{tabular}
    \label{table:1}
\end{table}

\begin{figure}[h]
    \centering
    \includegraphics[width=\linewidth]{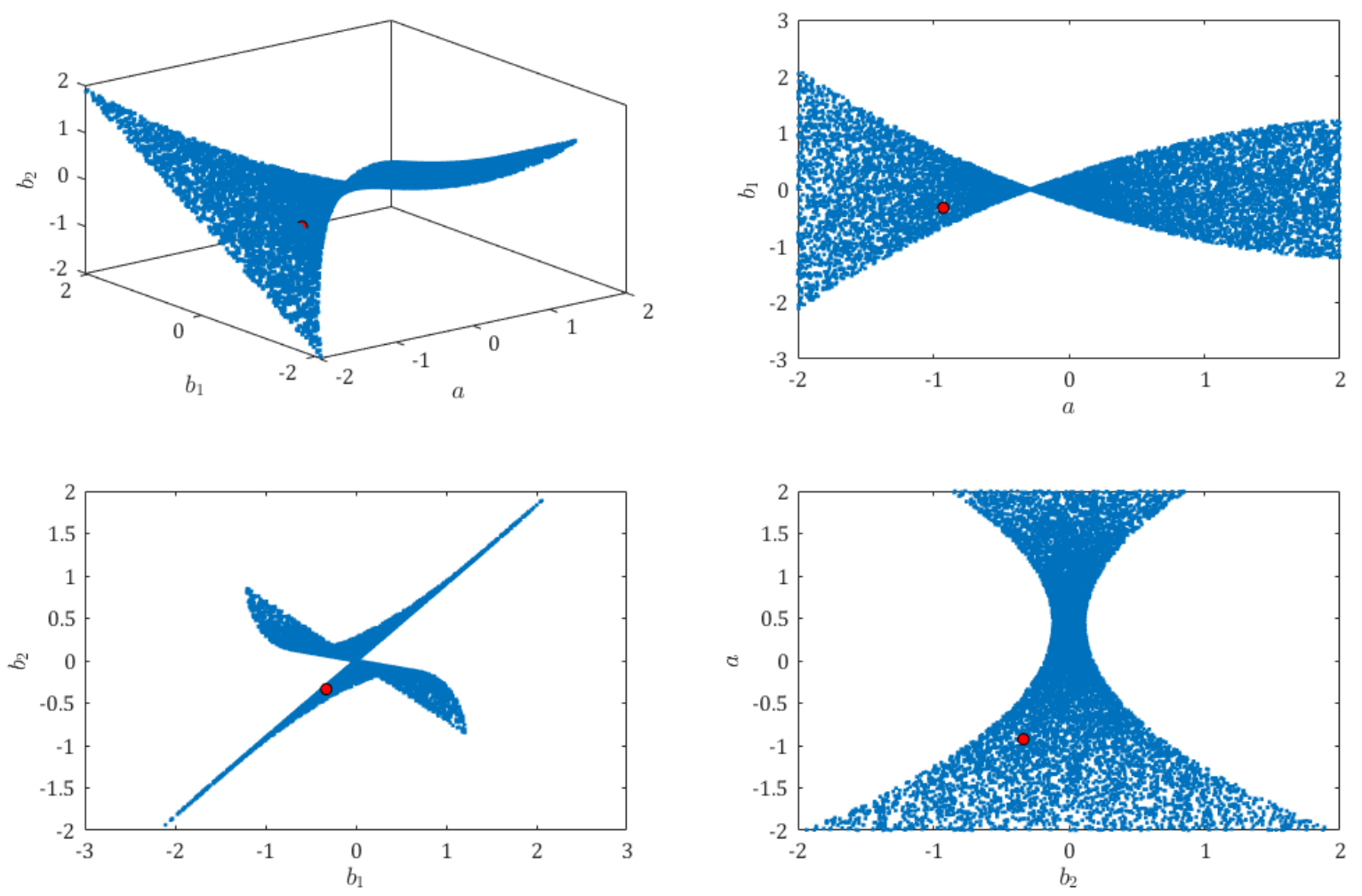}
    \caption{A set of controllable systems generating rank-deficient data for Example~\ref{ex:sec:V-1}.}
    \label{fig:1}
\end{figure}

\section{Conclusions}
\label{sec:VI}

In this work, we have defined the notion of a universal input as a signal that, when applied to any controllable system, results in input-output data that are sufficiently rich to parametrize all input-output trajectories of the system. We have proven that an input is universal if and only if it is persistently exciting of sufficiently high order. This result provides further insight into the fundamental lemma \cite{willems2005note}, from which only sufficient conditions for universality could be deduced. 

The study of universal inputs with respect to a subset of controllable systems is left for future work. In that case, universality is not necessarily equivalent to persistency of excitation. The study of universal inputs in the presence of measurement and process noise is also among the topics for future work.






\section*{References}

\bibliographystyle{IEEEtran}
\bibliography{biblo}

\end{document}